\documentclass[letterpaper, 10 pt, conference]{ieeeconf}  

\IEEEoverridecommandlockouts                              

\usepackage{graphics} 
\usepackage{epsfig} 
\usepackage{mathptmx} 
\usepackage{amsmath} 
\usepackage{amssymb}  

\newtheorem{theorem}{Theorem}

\title{\LARGE \bf
Stabilization of a Nonholonomic Car Model with Off-Hooked Trailers
}

\author{Alexander Zuyev$^{1,3}$ and Victoria Grushkovskaya$^{2,3}$
\thanks{$^{1}$Max Planck Institute for Dynamics of Complex Technical Systems, 39106 Magdeburg, Germany
        {\tt\small zuyev@mpi-magdeburg.mpg.de}}%
\thanks{$^{2}$Institute of Mathematics,
        University of Klagenfurt, 9020 Klagenfurt am W\"orthersee, Austria
        {\tt\small viktoriia.grushkovska@aau.at}}%
\thanks{$^{3}$Institute of Applied Mathematics \& Mechanics, National Academy of Sciences of Ukraine, 84116 Sloviansk, Ukraine}
}

\begin{document}

\maketitle
\thispagestyle{empty}
\pagestyle{empty}

\begin{abstract}
We consider a kinematic model of a controlled car with two trailers by assuming that each trailer is attached at some distance from the preceding axle (``off-hooked trailers''). For this model, we derive the transformation towards privileged coordinates and present the corresponding nilpotent quasihomogeneous approximate system. The components of this nilpotent approximation are written explicitly in terms of mechanical parameters of the original system. 
The constructed system does not satisfy the Brockett necessary stabilizability condition, and the design of time-varying feedback controllers with oscillating components is proposed. It is proved that these controllers ensure the exponential convergence of solutions to the trivial equilibrium, and simulation results are presented to illustrate the behavior of the closed-loop system.
\end{abstract}

\section{INTRODUCTION}

The stabilization problem for nonholonomic systems is a hot topic in mathematical control theory, which boasts rich theoretical content and significant applications in robotics.
To provide an overview without claiming to be comprehensive, we refer to the papers~\cite{zhang2022new,GZ23}  for surveys of some recent developments in this field. 

An important class of nonholonomic systems is represented by kinematic models of wheeled vehicles under the rolling without slipping conditions~\cite{luca1995modelling,grushkovskaya2018family}.
In particular, it is well-known that a kinematic model of a car with trailers can be transformed into a chained-form system~\cite{sordalen1993conversion}, provided that each subsequent trailer is attached to the center of the axle of the previous trailer.

In the paper~\cite{vendittelli1998nilpotent}, the authors considered the case of a car pulling two trailers, with each trailer hooked at a certain distance from the preceding wheel axle (``off-hooked trailers'').
It was noted in~\cite{vendittelli1998nilpotent} that the corresponding control system cannot be transformed into a chained form and does not satisfy the necessary and sufficient flatness conditions for two-input systems.
Although~\cite{vendittelli1998nilpotent} presented the general form of a nilpotent approximation for a car model with off-hooked trailers, it did not provide exact coordinate transformations or coefficients of vector fields. Moreover, the simulations were only conducted using open-loop controls for the case of unitary distances.

In the paper mentioned above, the construction of a homogeneous approximation is based on the approach described in~\cite{bellaiche2005geometry}. Another approximation scheme addresses the algebraic description of the time optimality problem as detailed in~\cite{sklyar2020construction}. Building on such an algebraic formulation, a method for constructing homogeneous approximations for control-affine systems has been proposed in~\cite{sklyar2022implementation} and implemented a Python-based web application. Although this application facilitates efficient computation of homogeneous approximations for systems that do not involve external parameters, the conversion to privileged coordinates still needs to be addressed, in addition to what the reported application covers.

The contribution of our present paper is twofold. First, we derive a mathematical model of a car with two off-hooked trailers. We assume that the length parameters of this model can be arbitrary. The resulting equations of motion are encapsulated in a driftless control-affine system with five states and two inputs.
To the best of our knowledge, the construction of nilpotent approximations for such parameterized models has not been explicitly treated so far. The considered class of systems does not satisfy the Brockett necessary stabilizability condition by a regular time-invariant feedback law. Thus, the second contribution of this paper is the development of stabilizing control algorithms with time-varying components. The exponential stabilization scheme with oscillating controllers is presented in Theorem~1. Finally, the performance of the constructed controllers is illustrated with numerical simulations.

\section{KINEMATIC MODEL OF A CAR WITH OFF-HOOKED TRAILERS}
\begin{figure}[t]
\hspace{-1em} \includegraphics[width=1\linewidth] {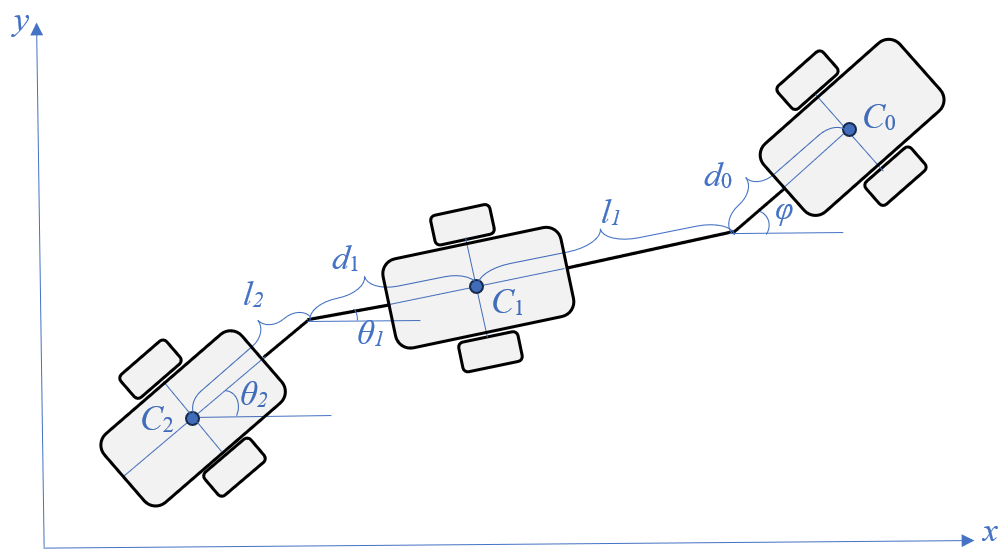}
 \caption{Car with two off-hooked trailers.}
\end{figure}
The rolling without slipping conditions take the form:
\begin{equation}\label{nonholonomic_constraints}
\begin{aligned}
& \dot x_{C_0} \sin \phi - \dot y_{C_0} \cos \phi =0,\\
& \dot x_{C_1} \sin \theta_1 - \dot y_{C_1} \cos \theta_1 =0,\\
& \dot x_{C_2} \sin \theta_2 - \dot y_{C_2} \cos \theta_2 =0,
\end{aligned}
\end{equation}
where $(x_{C_0},y_{C_0})$, $(x_{C_1},y_{C_1})$, and $(x_{C_2},y_{C_2})$ are the coordinates of points $C_0$, $C_1$, and $C_2$, respectively,
$$
\begin{aligned}
& x_{C_1} = x_{C_0}-d_0 \cos\phi - \ell_1 \cos \theta_1,\\
& y_{C_1} = y_{C_0}-d_0 \sin\phi - \ell_1 \sin \theta_1,\\
& x_{C_2} = x_{C_0}-d_0 \cos\phi - (d_1+\ell_1) \cos \theta_1 - \ell_2 \cos \theta_2,\\
& y_{C_2} = y_{C_0}-d_0 \sin\phi - (d_1+\ell_1) \sin \theta_1 - \ell_2 \sin \theta_2.
\end{aligned}
$$
Conditions~\eqref{nonholonomic_constraints} can be written as the driftless control-affine system
\begin{equation}\label{car_offhooking}
\dot x  = u_1 f_1(x) +u_2 f_2(x),\quad x\in{\mathbb R}^5,\; u\in{\mathbb R}^2
\end{equation}
with $x=(x_1,x_2,x_3,x_4,x_5)^\top$, $u=(u_1,u_2)^\top$, and
$$
f_1(x) = \begin{pmatrix}\cos x_3 \\ \sin x_3 \\ 0 \\ \frac{\sin(x_3-x_4)}{\ell_1} \\  f_{15}(x)\end{pmatrix},\;
f_2(x) = \begin{pmatrix}0 \\ 0 \\ 1 \\ -\frac{d_0\cos(x_3-x_4)}{\ell_1} \\ f_{25}(x)
\end{pmatrix},
$$
{\small
$$
\begin{aligned}
& f_{15}(x) = \frac{\sin(x_3-x_5)}{\ell_2} + \frac{(d_1+\ell_1)\sin(x_4-x_3)\cos(x_4-x_5)}{\ell_1 \ell_2},\\
& f_{25}(x) = \frac{d_0(d_1+\ell_1)\cos(x_3-x_4)\cos(x_4-x_5)}{\ell_1 \ell_2} - \frac{d_0\cos(x_3-x_5)}{\ell_2}.
\end{aligned}
$$}
Here, $x_1=x_{C_0}$, $x_2 = y_{C_0}$, $x_3=\phi$, $x_4=\theta_1$, $x_5 = \theta_2$,
$u_1$~represents the driving velocity, and $u_2$ denotes the steering velocity of the front axle, respectively.

For vector fields $f_i(x)$ and $f_j(x)$, we denote their Lie bracket as $[f_i,f_j](x)=\frac{\partial f_j(x)}{\partial x} f_i(x) - \frac{\partial f_i(x)}{\partial x} f_j(x)$. The system~\eqref{car_offhooking} satisfies the Lie Algebra Rank Condition~\cite{bellaiche2005geometry,NS} because the $5\times5$-matrix
$$
{\cal F}(x) = (f_1,f_2,[f_1,f_2],[f_1,[f_1,f_2]],[f_1,[f_1,[f_1,f_2]]])(x)
$$
is of full rank in a neighborhood of $x=0$. Indeed,
{\small
\begin{equation}\label{F_matrix}
{\cal F}(x) =  \begin{pmatrix}\cos x_3 & 0& \sin x_3 & 0 & 0 \\
\sin x_3 & 0& -\cos x_3 & 0 & 0 \\
0 & 1& 0 & 0 & 0 \\
\frac{\sin(x_3 - x_4)}{\ell_1} & -\frac{d_0 \cos(x_3 - x_4)}{\ell_1} & f_{34}(x)& f_{44}(x) & f_{54}(x) \\
f_{15}(x) & f_{25}(x) & f_{35}(x)& f_{45}(x) & f_{55}(x)
\end{pmatrix},
\end{equation}}
where the expressions for $f_{ij}(x)$ are given in the Appendix. Straightforward computations show that
\begin{equation}\label{rank5}
\det {\cal F}(0) = -\frac{(d_0+\ell_1)(d_0+\ell_2)(d_1+\ell_2)}{\ell_1^4 \ell_2^4} \neq 0.
\end{equation}

\section{CONSTRUCTION OF A WEIGHTED HOMOGENEOUS APPROXIMATION}
In this section, we apply the approach of~\cite{bellaiche2005geometry} to derive a homogeneous approximation of system~\eqref{car_offhooking} in a neighborhood of zero.
For this purpose, we first denote by ${\mathcal L}^1={\mathcal L}^1 (f_1,f_2)$ the set of all linear combinations of the vector fields $f_1$ and $f_2$ with real coefficients. Then, we consequently define ${\mathcal L}^s={\mathcal L}^{s} (f_1,f_2)$ as:
$$
{\mathcal L}^s={\mathcal L}^{s-1} + \sum_{i+j=s} [{\mathcal L}^{i},{\mathcal L}^{j}]\quad\text{for}\;s=2,3, ...,
$$
where ${\mathcal L}^0 =0$. For a point $p\in \mathbb R^5$ and nonnegative integer $s$, we denote the linear subspace
$$
L^s(p)=\{f(p)\,\vert\,f\in {\mathcal L}^{s}\}\subset {\mathbb R}^5.
$$
We have the following increasing sequence of linear subspaces because of condition~\eqref{rank5}:
$$
\{0\}\subset L^1(0)\subset L^2(0)\subset L^3(0)\subset L^4(0)={\mathbb R}^5,
$$
and $r=4$ is the degree of nonholonomy of system~\eqref{car_offhooking} at $p=0$.
The dimensions $n_s=\dim L^s(0)$ of the above subspaces are: $n_1=2$, $n_2=3$, $n_3=4$, $n_4=5$.
The algorithm of~\cite{bellaiche2005geometry} summarized in the following four steps.

In Step~1, we define the vector fields
$$
\begin{aligned}
& Y_1= f_1,\; Y_2= f_2,\; Y_3= [f_1,f_2],\; Y_4 = [f_1,[f_1,f_2]],\\
& Y_5 = [f_1,[f_1,[f_1,f_2]]],
\end{aligned}
$$
and note that $\{Y_1(0), Y_2(0), Y_3(0), Y_4(0), Y_5(0)\}$ is a basis of ${\mathbb R}^5$ due to condition~\eqref{rank5}.
Then, we introduce the privileged coordinates $y_1$, ..., $y_5$, such that $d y_1$, ..., $d y_5$ is a dual basis to $Y_1(0)$, ..., $Y_5(0)$, provided that the vector fields $Y_j$ are treated in the sense of differential operators.
It is easy to see that this can be formalized with the linear transformation
\begin{equation}\label{xy_transform_vector}
(y_1,y_2,y_3,y_4,y_5)^\top={\cal F}^{-1}(0)\,x,
\end{equation}
where ${\cal F}^{-1}(0)$ is the inverse matrix for~\eqref{F_matrix} evaluated at zero. In coordinate form, equation~\eqref{xy_transform_vector} is written as
\begin{equation}\label{xy_transform}
\begin{aligned}
y_1 =& x_1,\; y_2 =x_3,\; y_3 = -x_2,\\
y_4 =&(\ell_1 + \ell_2) x_2\\
& -\frac{d_0( d_0(\ell_1^2 + \ell_1 \ell_2 + \ell_2^2) + \ell_1 \ell_2(\ell_1 +\ell_2))}{(d_0 + \ell_1)(d_0+\ell_2)}x_3\\
& -\Bigl\{\frac{\ell_2^2 (\ell_1+\ell_2) (d_1 + \ell_1)}{(d_0+\ell_2) (d_1+\ell_2)} \\
&+ \frac{\ell_1^2 \ell_2(d_0d_1+(d_0 + d_1)\ell_1)  + d_0d_1\ell_1^3}{(d_0+\ell_1)(d_0+\ell_2) (d_1+\ell_2)} \Bigr\}x_4\\
&-\frac{\ell_2^4}{(d_0+\ell_2) (d_1+\ell_2)}x_5,\\
y_5 =&- \ell_1 \ell_2 x_2\\
& + \frac{d_0 \ell_1 \ell_2(\ell_1 \ell_2 + d_0(\ell_1 +\ell_2))} {(d_0 + \ell_1)(d_0+\ell_2)} x_3 \\
& + \Bigl\{\frac{\ell_1 \ell_2^3
(d_1+\ell_1)}{(d_0+\ell_2) (d_1+\ell_2)} \\
&+ \frac{\ell_1^2 \ell_2( \ell_2(d_0d_1+(d_0 + d_1)\ell_1) + d_0d_1 \ell_1
)
}{(d_0+\ell_1)(d_0+\ell_2) (d_1+\ell_2)} \Bigr\} x_4 \\
& + \frac{\ell_1 \ell_2^4}{(d_0+\ell_2) (d_1+\ell_2)} x_5.
\end{aligned}
\end{equation}

In Step~2, we assign the weight $w_j$ to each coordinate $y_j$ according to the following rule: $w_j$ is defined as the smallest integer $s\ge 1$ for which $d y_j$ is not identically zero on $L^s(0)$, where $j=\overline{1,5}$.
It means that 
$$w=(w_1,w_2,w_3,w_4,w_5)=(1,1,2,3,4).
$$

In Step~3, the nonlinear change of variables from $(y_1,...,y_5)$ to $(z_1,...,z_5)$ is defined recursively by~\cite[Theorem~1.5]{bellaiche2005geometry}:
\begin{equation}\label{zy_transformation}
z_q = y_q - \sum_{(\alpha,w) < w_q} \frac{1}{\alpha_1! \cdots \alpha_{q-1}!}(Y_1^{\alpha_1}\cdots Y_{q-1}^{\alpha_{q-1}} y_q)(0) z_1^{\alpha_1}\cdots z_{q-1}^{\alpha_{q-1}},
\end{equation}
where $(\alpha,w)=w_1\alpha_1 + ...+ w_5\alpha_5=\alpha_1+\alpha_2+2\alpha_3+3\alpha_4+4\alpha_5$, and the vector fields $Y_s$ are considered as differential operators.
Taking into account the transformation given in~\eqref{xy_transform}, the coordinate-wise representation of equation~\eqref{zy_transformation} is expressed as:
{\small
\begin{equation}\label{zy_coordinates}
\begin{aligned}
z_1 =& y_1 = x_1,\; z_2 = y_2 = x_3,\\
z_3 =& -x_2 + \sum_{\alpha_1+\alpha_2<2}\frac{1}{\alpha_1! \alpha_2!} (Y_1^{\alpha_1} Y_2^{\alpha_2} x_2)(0) z_1^{\alpha_1} z_2^{\alpha_2},\\
z_4 =& y_4 - \sum_{\alpha_1+\alpha_2+2\alpha_3<3}\frac{1}{\alpha_1! \alpha_2! \alpha_3!} (Y_1^{\alpha_1} Y_2^{\alpha_2}
Y_3^{\alpha_3} y_4)(0) z_1^{\alpha_1} z_2^{\alpha_2} z_3^{\alpha_3},\\
z_5 =& y_5 - \sum_{\alpha_1+\alpha_2+2\alpha_3+3\alpha_4<4}\frac{1}{\alpha_1! \cdots \alpha_4!} (Y_1^{\alpha_1} \cdots
Y_4^{\alpha_4} y_5)(0) z_1^{\alpha_1} \cdots z_4^{\alpha_4}.
\end{aligned}
\end{equation}}

Straightforward but cumbersome computations reveal that many terms on the right-hand side of~\eqref{zy_coordinates} vanish,
simplifying the transformation in~\eqref{zy_coordinates} to:
\begin{equation}\label{zy_coordinates_final}
\begin{aligned}
z_1 =& x_1,\; z_2 = x_3,\\
z_3 =& -x_2,\; z_4=y_4,\\
z_5=&y_5 - \beta x_3^3,
\end{aligned}
\end{equation}
where
$$
\beta = \frac{d_0\ell_2 \left[ d_0^2(d_1 + \ell_1)(\ell_1 - \ell_2 - d_1) + \ell_1^3(d_0 + \ell_2) \right] }{6(d_0+\ell_2)\ell_1^2}.
$$

In Step~4, we compute the time derivatives of $z_j$ given by~\eqref{zy_coordinates_final} along the trajectories of system~\eqref{car_offhooking}.
As a result, we obtain the following system of ordinary differential equations:
\begin{equation}\label{z_system_total}
\begin{aligned}
\dot z_1 =& u_1 + r_1(z,u),\\
\dot z_2 =& u_2,\\
\dot z_3 =& -z_2 u_1 + r_3(z,u),\\
\dot z_4 =& -z_3 u_1 - \alpha z_2^2 u_2 + r_4(z,u),\\
\dot z_5 =& (\kappa z_2^3-z_4) u_1 - \theta z_2 z_3 u_2 + r_5(z,u),
\end{aligned}
\end{equation}
where $r_1(z,u)=u_1 O(|z_2|^2)$, $r_3(z,u)=u_1 O(|z_2|^3)$ and, for bounded $u$, the remainder terms $r_4(z,u)$ and $r_5(z,u)$ are of degree at least $w_4$ and $w_5$ as $z\to 0$, respectively.
The parameters of system~\eqref{z_system_total} are defined as follows:
{\footnotesize
\begin{eqnarray}
\label{eq_alpha}
\alpha=&\frac{d_0^3[\ell_1^3 + \ell_1^2 \ell_2 - \ell_1\ell_2^2 - d_1 \ell_2(d_1+\ell_2)]
+d_0\ell_1^3 [d_0 (\ell_1 + 2\ell_2) + \ell_2(\ell_1 + \ell_2)]
}{2(d_0 + \ell_2)\ell_1^3},\\
\label{eq_kappa}
\kappa=&
\frac{d_0^2 \bigl[ \frac32 \ell_2^2(d_0+\ell_1)(d_1 + \ell_1) - \ell_2(\ell_1^3 + \frac12 d_0 \ell_1^2 - d_1\ell_1(d_0 + \frac32 d_1) - \frac32 d_0 d_1^2) + \frac12 d_0 d_1^2 \ell_1\bigr]}
{3(d_0 + \ell_2)\ell_1^3},\\
\label{eq_theta}
\theta =&\frac{d_0^2 \bigl[ \ell_2^2(d_1 + \ell_1)(d_0+\ell_1) + (d_0 d_1 + (d_0 + d_1)\ell_1)d_1\ell_2 + d_0 d_1^2 \ell_1\bigr]}{(d_0 + \ell_2)\ell_1^3}.
\end{eqnarray}}

By truncating the remainder terms in system~\eqref{z_system_total}, we express its approximation in a neighborhood of $z=0$ in the form
\begin{equation}\label{sys_nilpotent}
\dot z = u_1 g_1(z) + u_2 g_2(z),\quad z\in {\mathbb R}^5,\; u\in {\mathbb R}^2,
\end{equation}
$$
z=
\begin{pmatrix}
z_1 \\ z_2 \\ z_3 \\ z_4 \\ z_5
\end{pmatrix},\;
g_1(z)=
\begin{pmatrix}
1 \\ 0 \\ -z_2 \\ -z_3 \\ \kappa z_2^3-z_4
\end{pmatrix},\;
g_2(z) = \begin{pmatrix}
0 \\ 1 \\ 0 \\ -\alpha z_2^2 \\ -\theta z_2 z_3
\end{pmatrix}.
$$
The right-hand side of the differential equation for $\dot z_j$ in~\eqref{sys_nilpotent} is a weighted homogeneous polynomial of degree $w_j-1$ for each $j=1,2,3,4,5$.

\section{DESIGN OF TIME-VARYING FEEDBACK LAWS}

To stabilize the trivial equilibrium of system~\eqref{sys_nilpotent}, we will extend the approach of~\cite{ZGB16} for the case of higher degree nonholonomic systems.

First, we evaluate the matrix
$$
{\cal G}(z)=\bigl(g_1,g_2,[g_1,g_2],[g_1,[g_1,g_2]],[g_1,[g_1,[g_1,g_2]]]\bigr)(z),
$$
the columns of which are composed of the Lie brackets of the vector fields $g_1(z)$ and $g_2(z)$. 
We have
\begin{equation}\label{G_matrix}
{\cal G}(z)=\begin{pmatrix}
1 & 0 & 0 & 0 & 0\\
0 & 1 & 0 & 0 & 0\\
-z_2 & 0 & 1 & 0 & 0\\
-z_3 & -\alpha z_2^2 & 0 & 1 & 0\\
\kappa z_2^3 - z_4 & -\theta z_2 z_3 & -(\alpha + 3 \kappa - \theta) z_2^2 & 0 & 1\\
\end{pmatrix},
\end{equation}
ensuring that the Lie Algebra Rank Condition is satisfied: $\det {\cal G}(z)= 1$ for all $z\in {\mathbb R}^5$.
Hence, by the Chow--Rashevskii theorem~\cite{NS}, system~\eqref{sys_nilpotent} is controllable.

According to the control design scheme developed in~\cite{ZGB16}, we construct a family of time-varying feedback controllers.
These controllers are in the form of $\varepsilon$-periodic trigonometric polynomials in $t$, with coefficients that depend on $z$.
In the sequel, we use the following parametrization for the controllers:
\begin{equation}\label{controls_original}
\begin{aligned}
u_1^\varepsilon&(t,z) =  a_1(z) + \left(2k_{12}\omega |a_{12}(z)|\right)^{1/2} \cos (k_{12}\omega t)\\
& + (2k_{112}\omega)^{2/3} |a_{112}(z)|^{1/3}\cos (k_{112}\omega t)\\
&+ 2(k_{1112}\omega)^{3/4} |a_{1112}(z)|^{1/4}\cos (k_{1112}\omega t),\\
u_2^\varepsilon&(t,x) =  a_2(z) + \left(2k_{12}\omega |a_{12}(z)|\right)^{1/2} ({\rm sign}\,a_{12}(z))\, \sin (k_{12}\omega t)\\
& - 2(2k_{112}\omega)^{2/3} |a_{112}(z)|^{1/3}({\rm sign}\,a_{112}(z))\,\cos (2k_{112}\omega t)\\
&- 6(k_{1112}\omega)^{3/4} |a_{1112}(z)|^{1/4}({\rm sign}\,a_{1112}(z))\,\sin (3 k_{1112}\omega t),\\
\end{aligned}
\end{equation}
where $\omega=2\pi/\varepsilon$,
 and the positive integers $k_{12}$, $k_{112}$, $k_{1112}$ should be chosen
to avoid resonances up to the fourth order. To be precise, we recall that a tuple of numbers  $(k_1,\dots,k_s)$ is in \emph{resonance of order $K$} ($K\in\mathbb N$) if there exist setwise coprime integers  $(c_1,\dots,c_s )$ such that 
$\sum_{i=1}^s|c_i|=K$ and $\sum_{i=1}^sc_sk_s=0$.

The vector of state-dependent coefficients $a(z)=(a_1(z),a_2(z),a_{12}(z),a_{112}(z),a_{1112}(z))^\top$ is computed in terms of the inverse matrix for~\eqref{G_matrix}:
\begin{equation}\label{a_form}
a(z)=-{\cal G}^{-1}(z)\left(\frac{\partial Q(z)}{\partial z_1}, \frac{\partial Q(z)}{\partial z_2} \frac{\partial Q(z)}{\partial z_3}, \frac{\partial Q(z)}{\partial z_4},\frac{\partial Q(z)}{\partial z_5}\right)^\top,
\end{equation}
where $Q(z)$ is a positive definite quadratic form. For further computations, we take $Q(z)=\frac{\gamma}{2}(z_1^2+z_2^2+z_3^2+z_4^2+z_5^2)$, where $\gamma$ is a positive constant.
Then the components of $a(z)$ take the following form:
\begin{equation}\label{a_components}
\begin{aligned}
a_1(z) &= -\gamma z_1,\\
a_2(z) &= -\gamma z_2,\\
a_{12}(z)&= -\gamma(z_1 z_2 + z_3),\\ 
a_{112}(z)&=  -\gamma( z_1 z_3 + z_4 + \alpha z_2^3),\\
a_{1112}(z)&=  -\gamma ( (\alpha + 2\kappa - \theta)z_1 z_2^3 \\
&\qquad+ (\alpha + 3\kappa)z_2^2 z_3 + z_1 z_4 + z_5).
\end{aligned}
\end{equation}
We will consider system~\eqref{sys_nilpotent} in a $\Delta$-neighborhood of zero, denoted by $B_\Delta(0)=\{z\in{\mathbb R}^5\,:\,\|z\|<\Delta\}$.

The contribution of this section establishes the exponential convergence property for the solutions of system~\eqref{sys_nilpotent} with the controls~\eqref{controls_original}--\eqref{a_components} defined in the sense of sampling.
To formulate the main result, we consider the partition of $t\in[0,+\infty)$ into intervals
$$
0=t_0 < t_1=\varepsilon < t_2= 2\varepsilon <...< t_j= j\varepsilon <...\,
$$
where $\varepsilon>0$ is given. We call a vector function $z:[0,+\infty)\to {\mathbb R}^5$ a {\em $\pi_\varepsilon$-solution of system~\eqref{sys_nilpotent} with the feedback law~\eqref{controls_original}} (see, e.g.,~\cite{ZGB16}), if $z(t)$ is continuous on $[0,+\infty)$ and $z(t)$ satisfies the differential equations
$$
\dot z(t) = u_1^\varepsilon (t,z(t_j)) g_1(z(t)) + u_2^\varepsilon (t,z(t_j)) g_2(z(t)),\;
t\in [t_j,j_{j+1}),
$$
for each $j=0,1,2,...$ . Thus, a $\pi_\varepsilon$-solution $z(t)$ satisfies system~\eqref{sys_nilpotent} with piecewise-continuous controls of the form
\begin{equation}
    \label{controls_sampling}
u_s=u_s^\varepsilon(t,z(t_j)),\quad t\in[j\varepsilon,(j+1)\varepsilon).
\end{equation}

\begin{figure*}[thpb]
 \includegraphics[width=1\linewidth] {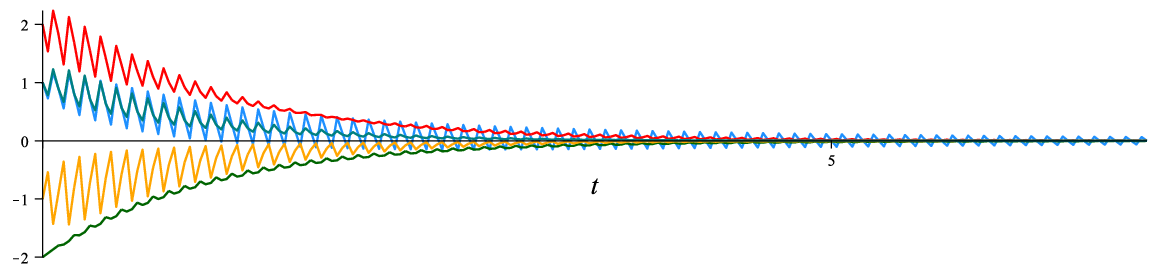}
\includegraphics[width=1\linewidth] {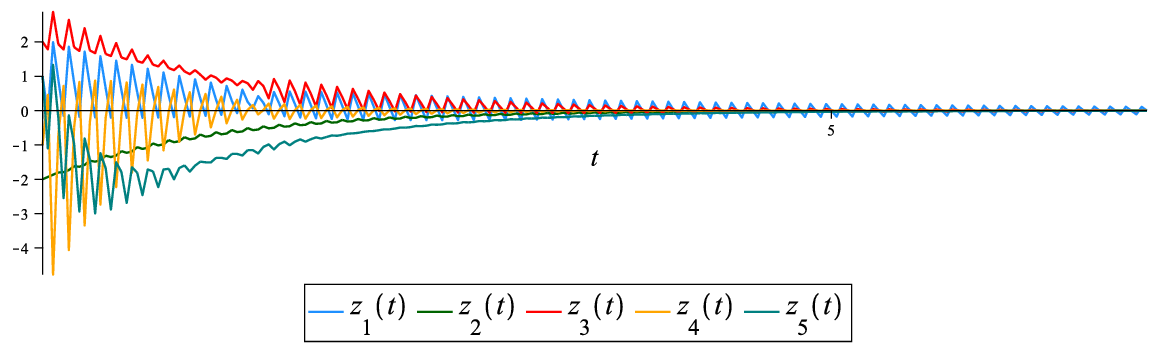}
 \caption{Time plots of the components of solutions to system~\eqref{sys_nilpotent} with controls~\eqref{controls_sampling} and parameters~\eqref{parameters1} (top) and~\eqref{parameters2} (bottom).} 
\end{figure*}
\begin{theorem} Assume that the frequency multipliers $k_{12}$, $k_{112}$, $k_{1112}$ in~\eqref{controls_original} satisfy the condition:
\begin{itemize}
\item[$C1)$] there are no resonances of order up to $4$ in each of the following tuples: $(k_{12},\;k_{112}, k_{1112})$, $(k_{12},\;2 k_{112}, k_{1112})$, $(k_{12},\;k_{112}, 3k_{1112})$, and $(k_{12},\;2 k_{112}, 3k_{1112})$.
\end{itemize}

Then, for any $\Delta>0$, there exists an $\bar\varepsilon>0$ such that for any $\varepsilon\in(0,\bar\varepsilon)$ and any $z^0\in B_\Delta(0)$, the corresponding $\pi_\varepsilon$-solution $z(t)$ of system~\eqref{sys_nilpotent} with the initial data $z(0)=z^0$ and feedback law~\eqref{controls_original} converges exponentially to zero as $t\to+\infty$.
\end{theorem}
\begin{proof}
To prove this theorem, we adopt the approach outlined in~\cite{ZGB16,GZ23}. However, the results presented there cannot be directly applied here, as they are intended for nonholonomic systems of degree two and three. 

Given any $\Delta>0$ and any $z^0\in D=B_\Delta(0)$, let $z(t)$ be a solution of system~\eqref{sys_nilpotent} with the initial condition $z(0)=z^0$ and control~\eqref{controls_sampling}, and let $\varepsilon_0>0$ be such that $z(t)\in D$ for all $t\in[0,\varepsilon_0]$.  Without loss of generality, we assume $\varepsilon_0\le 1$.

First, let us derive several  preliminary estimates. The integral representation of solutions yields
\begin{equation}\label{aprior1}
\begin{aligned}
\|z(t)-z^0\|&\le\int\limits_0^\varepsilon\sum_{s=1,2}\|g_s(z(\tau))\||u_s^\varepsilon(z_0,\tau)|d\tau\\
&\le\sqrt 2
\max\limits_{s=1,2,0\le t\le\varepsilon}|u_s^\varepsilon(z_0,\tau)|\int\limits_0^\varepsilon\|g(z(\tau))\|d\tau.
\end{aligned}
\end{equation}
We estimate 
$$
\begin{aligned}
\|g(z)\|^2&=\|g_1(z)\|^2+\|g_2(z)\|^2\\
&\le 2+\|z\|^2+c_{g1}\|z\|^4+c_{g2}\|z\|^6
\end{aligned}
$$
with $c_{g1}=\frac{1}{2}\max\left\{\kappa,\theta,3\kappa+\theta+2\alpha^2\right\}=\frac{1}{2}(3\kappa+\theta+2\alpha^2)$, $c_{g2}=\kappa^2$. Besides, one can show that there exists a $c_u>0$
such that
\begin{equation}\label{est_u}
\begin{aligned}
 |u_s^\varepsilon(z^0,\tau)|\le c_u\varepsilon^{-3/4}\|z^0\|^{1/4},
\end{aligned}
\end{equation}
for all  $z^0\in D$ and $\varepsilon\in[0,1]$.  The estimate~\eqref{est_u} can be obtained from~\eqref{a_components} using H\"older's inequality. We omit the explicit expression for $c_u$ because of space limits. Thus, for all  $z^0\in D$, $\varepsilon\in(0,\varepsilon_0)$, and $t\in[0,\varepsilon]$,
\begin{equation}
    \label{est_z}
\|z(t)-z^0\|\le c_z\varepsilon^{1/4}\|z^0\|^{1/4},
\end{equation}
with some $c_z>0$. 
Next, consider the Chen--Fliess series expansion~\cite{NS} of $z(t)$ at $t=\varepsilon$:
\begin{equation}\label{ChenFliess0}
    \begin{aligned}
       z&(\varepsilon)=z^0+\sum_{s=1,2} g_s(z^0)\int\limits_0^\varepsilon u_s^\varepsilon(z^0,\tau)\\
       +&\sum_{l=2}^4\sum_{s_1,\dots,s_l=1}^2L_{g_{s_l}}\dots L_{g_{s_2}}g_{s_1}(z^0)\\
       &\times \int\limits_0^\varepsilon\int\limits_0^{\tau_1}\dots \int\limits_0^{\tau_{l-1}}u_{s_1}^\varepsilon(z^0,\tau_1)\dots u_{s_l}^\varepsilon(z^0,\tau_l)d\tau_1\dots d\tau_l\\
       +&\Phi(z^0,\varepsilon),
    \end{aligned}
\end{equation}
where $\Phi$ is the remainder, of the Chen--Fliess series expansion,
$$
\begin{aligned}
\Phi(z^0,\varepsilon)  =  \sum_{l=2}^4&\sum_{s_1,\dots,s_l=1}^2  \int\limits_0^\varepsilon\int\limits_0^{\tau_1}\dots \int\limits_0^{\tau_{l-1}} L_{g_{s_l}}\dots L_{g_{s_2}}g_{s_1}(z(\tau_l))\\
       &\times u_{s_1}^\varepsilon(z^0,\tau_1)\dots u_{s_l}^\varepsilon(z^0,\tau_l)d\tau_1\dots d\tau_l.
\end{aligned}
$$
As the $j$-th component of each vector field $g_1$ and $g_2$ in system~\eqref{sys_nilpotent} is a weighted homogeneous polynomial of degree $w_j-1$~\cite{bellaiche2005geometry} , the vector fields $g_1$ and $g_2$ are homogeneous of degree $-1$ with respect to the dilation
$$
(z_1,z_2,z_3,z_4,z_5)\mapsto (\lambda z_1, \lambda z_2, \lambda^2 z_3, \lambda^3 z_4,\lambda^4 z_5).
$$
This implies that the Lie algebra generated by $\{g_1,g_2\}$ is nilpotent of step $r=4$ (see, e.g.,~\cite{grayson1989vector}). Consequently, any iterated Lie bracket of $g_1$ and $g_2$ with a length greater than $4$ is zero.

Because of the nilpotent property of system~\eqref{sys_nilpotent}, the Chen--Fliess series expansion~\eqref{ChenFliess0}~ of $z(t)$ on the interval $t\in [0,\varepsilon]$ does not contain the terms with iterated integrals of $u_1$ and $u_2$ of length higher than $4$. Hence, $\Phi(z_0,\varepsilon)\equiv 0$ in~\eqref{ChenFliess0}.
We compute the Chen--Fliess  expansion~\eqref{ChenFliess0} explicitly by assuming that condition $C1)$ holds:
\begin{equation}\label{ChenFliess}
\begin{aligned}
z(\varepsilon)&=z^0+\varepsilon\Bigl\{a_1(z^0)g_1(z^0)+a_2(z^0)g_2(z^0)\\
&+a_{12}(z^0)[g_1,g_2](z^0) + a_{112}(z^0)[g_1,[g_1,g_2]](z^0)\\
&+ a_{1112}(z^0)[g_1,[g_1,[g_1,g_2]]](z^0)\Bigr\} + \varepsilon^{4/3}\|z^0\|^{5/4}R(z^0,\varepsilon)\\
&=z^0+\varepsilon\mathcal G(z^0)a(z^0)+\varepsilon^{4/3}\|z^0\|^{5/4}R(z^0,\varepsilon),
\end{aligned}
\end{equation}
where the function $R(z^0,\varepsilon)$ is bounded,
i.e. $\|R(z^0,\varepsilon)\| \le M_D$ for all $z^0\in B_\Delta(0)$ and $\varepsilon\in (0,1]$, provided that the domain $D$ is bounded. The explicit expression of $R(z^0,\varepsilon)$ is omitted because of the space limits. Taking into account~\eqref{a_form}, we conclude that
\begin{equation}
    \label{z_eps}
    z(\varepsilon)=(1-\gamma\varepsilon)z^0+\varepsilon^{4/3}\|z^0\|^{5/4}R(z^0,\varepsilon).
\end{equation}
For any $\bar\gamma\in(0,\gamma)$, let us put $\varepsilon_2=\min\{\frac{1}{\bar\gamma},\left(\frac{\gamma-\bar\gamma}{\Delta^{1/4}M_D}\right)^{3}\}$. Then, for any $\varepsilon\in(0,\min\{\varepsilon_0,\varepsilon_1,\varepsilon_2\})$, 
$$
\begin{aligned}
    \|z(\varepsilon)\|&\le (1-\gamma\varepsilon)\|z^0\|+\varepsilon^{4/3}M_D\|z^0\|^{5/4}\\
    &\le (1-\varepsilon(\gamma-\varepsilon^{1/3}\Delta^{1/4}M_D))\|z^0\|\\
    &\le (1-\varepsilon\bar\gamma)\|z^0\|\le \|z^0\|e^{-\bar\gamma\varepsilon}.
\end{aligned}
$$
Thus, $\|z(\varepsilon)\|\le\|z^0\|<\Delta$, and we may repeat the above reasoning. Consequently, we arrive at the following conclusion:  
$$
\begin{aligned}
 \| z(j\varepsilon)\|&\le (1-\varepsilon\bar\gamma)\|z((j-1)\varepsilon)\| \le (1-\varepsilon\bar\gamma)^2\|z((j-2)\varepsilon)\|\\
  &\le\dots\le  (1-\varepsilon\bar\gamma)^j\|z^0\|\le\|z^0\|e^{-\bar\gamma j\varepsilon},
\end{aligned}
$$
and
$$
\begin{aligned}
    \|z(t)-z^{j\varepsilon}\|&\le c_z\varepsilon^{1/4}\|z(j\varepsilon)\|^{1/4}\\
    &\le  c_z\varepsilon^{1/4}\|z^0\|^{1/4}e^{-\bar\gamma j\varepsilon/4}.
\end{aligned}
$$
For an arbitrary $t\ge0$,  let $j=\left[\frac{t}{\varepsilon}\right]$ denotes the integer part of $\frac{t}{\varepsilon}$. Then $t\in[j\varepsilon,(j+1)\varepsilon]$, and
$$
\begin{aligned}
    \|z(t)\|&\le \|z(j\varepsilon)\|+\|z(t)-z^{j\varepsilon}\|\\
&\le c_{zt}\|z^0\|^{1/4}e^{-\bar\gamma t/4},
\end{aligned}
$$
where $z_{zt}=e^{\bar\gamma\varepsilon/4}(\Delta^{3/4}+c_z)$, what proves the theorem. 
\end{proof}

\section{SIMULATION RESULTS}

For the numerical simulations, we consider two choices of physical length parameters:
$$
   d_0=d_1=0.1,\, \ell_1=\ell_2=1,
$$ 
and
$$
   d_0=d_1= \ell_1=\ell_2=1,
$$
which correspond to 
\begin{equation}
    \label{parameters1}
\alpha=0.10495, \; \kappa=0.0024,\; \theta=0.0112,
\end{equation}
and
\begin{equation}
    \label{parameters2}
\alpha=1, \; \kappa=1.5,\; \theta=4,
\end{equation}
respectively.
In both cases, we put
$  \varepsilon=0.1$, $\gamma=1$, 
$k_{12}=8$, $k_{112}=6$, $k_{1112}=5$, and initialize the solutions at $z^0=(1,-2,2,-1,1)^\top$.
The results of numerical simulations are shown in Fig.~1, which illustrates the asymptotic  convergence of the solutions of system~\eqref{sys_nilpotent} to zero.  

\section{CONCLUSIONS}

As one can see from the simulation results, the proposed time-varying controllers steer the solutions of system~\eqref{sys_nilpotent} to zero for both sets of length parameters~\eqref{parameters1} and~\eqref{parameters2}. However, the time plots for the longer hook length~\eqref{parameters2} exhibit a larger overshoot in the transient behavior. The analytic expressions~\eqref{controls_original}--\eqref{a_components} are valid for arbitrary values of mechanical parameters $\alpha$, $\kappa$, $\theta$, and arbitrary design parameters $\varepsilon>0$, $\gamma>0$. Although the exponential convergence result is limited to the nilpotent approximation~\eqref{sys_nilpotent}, we plan to investigate the applicability of our control design methodology to the original system~\eqref{car_offhooking} in future studies.

\newpage



\section*{APPENDIX}

The components $f_{34}$, $f_{44}$, $f_{54}$, $f_{35}$, $f_{45}$, and $f_{55}$ of the matrix ${\cal F}(x)$ are defined by the following formulas:
{\small
$$
\begin{aligned}
f_{34}(x) &=  -\frac{d_0+\ell_1\cos(x_3 - x_4)}{\ell_1^2},\\
f_{44}(x) &= -\frac{\ell_1+d_0 \cos(x_3 - x_4)}{\ell_1^3},\\
f_{54}(x) &= -\frac{d_0 + \ell_1 \cos(x_3 - x_4)}{\ell_1^4},\\
f_{35}(x) &= \frac{1}{\ell_1^2 \ell_2^2}\Bigl\{
(d_1 + \ell_1) \cos(x_4 - x_5) \bigl[
\ell_1 (d_0 \cos(x_3 - x_5) + \ell_2) \\
&\times \cos(x_3 - x_4) + d_0 (\ell_1 \sin(x_3 - x_4)\sin(x_3 - x_5) + \ell_2)
\bigr]\\
& - \ell_1 \bigl[
d_0 (d_1 + \ell_1)\cos(x_3 - x_4) \sin(x_5 - x_3)\sin(x_4 - x_5) \\
& + d_0 \ell_1 + \cos(x_3 - x_5) \Bigl(\ell_1 \ell_2 \\
& + d_0 (d_1 + \ell_1)\sin(x_3 - x_4) \sin(x_4 - x_5)
\Bigr)
\bigr]
\Bigr\},\\
f_{45}(x) &= \frac{1}{\ell_1^3 \ell_2^3}\Bigl\{
d_0 \ell_1 (d_1 + \ell_1)^2\cos^2(x_4 - x_5)\sin(x_3 - x_4)\\
&\times
(\sin(x_3 - x_4)\cos(x_3 - x_5) + \sin(x_5 - x_3)\cos(x_3 - x_4))\\
& + (d_1 + \ell_1)\cos(x_4 - x_5)\Bigl[
d_0 \ell_1^2\cos(x_3 - x_4)\cos^2(x_3 - x_5) \\
& + \ell_1 \cos(x_3 - x_5)
\bigl(
\cos(x_3 - x_4)(d_0(d_1 + \ell_1) \sin(x_4 - x_5)\\
& \times\sin(x_3 - x_4) + \ell_1 \ell_2) + d_0 (\ell_2 + \ell_1\sin(x_3 - x_4)\sin(x_3 - x_5) )
\bigr) \\
&
 - d_0\ell_1(d_1 + \ell_1) \sin(x_3 - x_5)\sin(x_4 - x_5)\cos^2(x_3 - x_4)\\
&
 + d_0\ell_2^2 \cos(x_3 - x_4) +
 \ell_1\bigl( \ell_2^2 + \sin(x_3 - x_5) \\
&
 \times (d_0 (d_1 + \ell_1) \sin(x_4 - x_5) + \ell_1 \ell_2 \sin(x_3 - x_4))
\bigr)
\Bigr]\\
& + \ell_1\Bigl[
d_0 \ell_1 (d_1+\ell_1) \sin(x_4-x_3) \sin(x_4-x_5)\cos^2(x_3-x_5)\\
& +\cos(x_3-x_5) \bigl(
d_0 (d_1+\ell_1)^2 \cos^2(x_3 - x_4)  \\
& + d_0 \ell_1 (d_1+\ell_1) \sin(x_3-x_5) \sin(x_4-x_5)  \cos(x_3-x_4)-\ell_1 \ell_2 \\
&
\times(d_1+\ell_1) \sin(x_3-x_4)  \sin(x_4-x_5) - 2 d_0 (\ell_1^2+d_1 \ell_1+\frac12 {d_1^2})
\bigr) \\
&
+(d_1+\ell_1) (\ell_1 \ell_2 \sin(x_4-x_5)+d_0 (d_1+\ell_1) \sin(x_3-x_4) ) \\
& \times \sin(x_3-x_5) \cos(x_3-x_4)
+d_0 (d_1+\ell_1) \ell_2 \sin(x_4-x_5) \\
& \times \sin(x_3-x_5) -\ell_1( d_0 (d_1+\ell_1) \sin(x_3-x_4) \sin(x_4-x_5)\\
& +\ell_1 \ell_2 )
\Bigr]\Bigl\},\\
f_{55}(x) &= \frac{1}{\ell_1^4 \ell_2^4}\Bigl\{
-d_0 \ell_1 (d_1 + \ell_1)^2 \cos^3(x_3 - x_4) \\
& \times \bigl[((d_1 + \ell_1)\cos(x_4 - x_5) + \ell_2)\cos(x_3 - x_5)\\
& +(d_1 + \ell_1) \sin(x_3 - x_5)\sin(x_4 - x_5) \bigr]- (d_1 + \ell_1)\ell_1\\
& \times \cos^2(x_3 - x_4) \bigl[
4d_0 \ell_1 (d_1 + \ell_1)(\cos^2(x_4 - x_5) - \frac12) \\
&\times \cos^2(x_3 - x_5) + (d_1 + \ell_1)(-\ell_1 \ell_2 \sin^2(x_4 - x_5)\\
& + 4 d_0 \ell_1 \sin(x_3 - x_5)\sin(x_4 - x_5)\cos(x_4 - x_5)\\
& - d_0 (d_1 + \ell_1)\sin(x_4 - x_5)\sin(x_3 - x_4))\cos(x_3 - x_5)\\
&- 2 d_0 \ell_1(d_1 + \ell_1)\cos^2(x_4 - x_5) + ((\ell_1 \ell_2\sin(x_4 - x_5) \\
&+ d_0(d_1 + \ell_1)\sin(x_3 - x_4))(d_1 + \ell_1)\sin(x_3 - x_5) - d_0 \ell_1 \ell_2)\\
&\times \cos(x_4 - x_5) - d_0(d_1 + \ell_1)(\ell_1-\ell_2\sin(x_3 - x_4)\\
&\times\sin(x_3 - x_5))\bigr]
- (d_1 + \ell_1)\ell_1 \cos(x_3 - x_4) \\
&\times \bigl[
-\ell_1( 4 d_0(d_1 + \ell_1) \sin(x_4 - x_5)\sin(x_3 - x_4) + \ell_1\ell_2)\\
&\times\cos(x_4 - x_5)\cos^2(x_3 - x_5) + (4\ell_1 d_0(d_1 + \ell_1) \\
&\times\sin(x_3 - x_4)\sin(x_3 - x_5)\cos^2(x_4 - x_5) \\
\end{aligned}
$$
$$
\begin{aligned}
&
+ (-\ell_1\ell_2(d_1 + \ell_1)\sin(x_3 - x_4)\sin(x_4 - x_5) \\
&- d_0(2\ell_1^2 + 2d_1\ell_1 + \ell_2^2 + d_1^2))\cos(x_4 - x_5) \\
&- 2\ell_1(\sin(x_3 - x_4)d_0(d_1 + \ell_1) + \frac12 \ell_1\ell_2\sin(x_4 - x_5))\\
&\times \sin(x_3 - x_5) - \ell_2d_0(d_1 + \ell_1))\cos(x_3 - x_5) + \ell_1\ell_2(d_1 + \ell_1)\\
&\times\sin(x_3 - x_4)\sin(x_3 - x_5)\cos^2(x_4 - x_5) + (2\ell_1d_0(d_1 + \ell_1)\\
&\times \sin(x_3 - x_4)\sin(x_4 - x_5) - \ell_2^3)\cos(x_4 - x_5) - (\ell_1\ell_2(d_1 + \ell_1)\\
&\times \sin(x_3 - x_4) + d_0(2\ell_1^2 + 2d_1\ell_1 + \ell_2^2 + d_1^2)\sin(x_4 - x_5))\\
&\times \sin(x_3 - x_5) - d_0\ell_1\ell_2\sin(x_3 - x_4)\sin(x_4 - x_5)
\bigr] \\
& +\ell_1^2 (4 d_0 (d_1 + \ell_1) \cos(x_4 - x_5)^2 + d_0 \ell_2 \cos(x_4 - x_5) \\
&- \ell_1 \sin(x_3 - x_4) \sin(x_4 - x_5) \ell_2 - 2 d_0 (d_1 + \ell_1)) \\
&\times (d_1 + \ell_1) \cos^2(x_3 - x_5) + \ell_1 \cos(x_3 - x_5) \\
&\times \bigl[
\ell_1 \ell_2 (d_1 + \ell_1)^2 \cos^2(x_4 - x_5) + \ell_1 ((\ell_1  \ell_2\sin(x_3 - x_4) \\
& + 4 d_0 (d_1 + \ell_1) \sin(x_4 - x_5) ) \sin(x_3 - x_5) + \ell_2^2) (d_1 + \ell_1) \\
&\times \cos(x_4 - x_5) + \ell_1 \ell_2 d_0 (d_1 + \ell_1) \sin(x_4 - x_5)  \sin(x_3 - x_5) \\
& - d_0(d_1 + \ell_1)(4\ell_1^2 + 2 d_1 \ell_1 + \ell_2^2 + d_1^2)  \sin(x_4 - x_5)  \sin(x_3 - x_4)\\
& - \ell_1 \ell_2 (2\ell_1^2 + 2d_1 \ell_1 + d_1^2)\bigr]-2 \ell_1^2 d_0 (d_1 + \ell_1)^2 \cos^2(x_4 - x_5) \\
&+ (\ell_1 (d_0 (4\ell_1^2 + 2 d_1 \ell_1 + \ell_2^2 + d_1^2) \sin(x_3 - x_4) \\
& + \ell_1 \sin(x_4 - x_5) \ell_2 (d_1 + \ell_1)) \sin(x_3 - x_5) -\ell_1^2 \ell_2 d_0\\
& + \ell_2^3 d_0) (d_1 + \ell_1) \cos(x_4 - x_5)-\ell_1 \bigl[
-(\ell_1 \ell_2 \sin(x_4 - x_5) \\
& + d_0 (d_1 + \ell_1) \sin(x_3 - x_4) ) \ell_2 (d_1 + \ell_1) \sin(x_3 - x_5) \\
& + \ell_1 (\ell_1 \ell_2 (d_1 + \ell_1) \sin(x_3 - x_4)  \sin(x_4 - x_5) \\
& + d_0 (2\ell_1^2 + 2d_1 \ell_1 + d_1^2)) \bigr]
\Bigl\}.
\end{aligned}
$$}

\bibliographystyle{ieeetr}
\bibliography{biblio_nonh}
\end{document}